\documentclass[a4paper, 11pt, parskip=half]{amsart}

\usepackage{customtemplate1}
\usepackage{mathtools}
\usepackage{tabularx}
\usepackage{framed}
\usepackage{array}
\usepackage{shortcuts}
\usepackage{stmaryrd}
\usepackage{bbm}
\begin{document}
\title{
    Determining newforms via arithmetic relations among Fourier coefficients}
	\date{}
	\author{Arvind Kumar, Moni Kumari, Prabhat Kumar Mishra}
    \address{Department of Mathematics, Indian Institute of Technology Jammu, Jagti, PO Nagrota, NH-44 Jammu 181221, J \& K, India\vspace*{-5pt}}
    \email{arvind.kumar@iitjammu.ac.in\vspace*{-6pt}}
	\email{moni.kumari@iitjammu.ac.in\vspace*{-6pt}}
	\email{2022rma0025@iitjammu.ac.in\vspace*{-6pt}}

\subjclass[2020]{Primary: 11F11; Secondary: 11F30}
\keywords{Newforms, Fourier coefficients, effective joint Sato--Tate distribution, multiplicity one theorem, twist-equivalence, density of primes}

\begin{abstract}
We investigate the distribution of primes satisfying arithmetic inequalities involving the Fourier coefficients of two non-CM newforms at prime powers. More precisely, we establish asymptotic formulas for the number of primes for which the differences, products, and ratios of the Fourier coefficients satisfy prescribed inequalities, together with explicit estimates for the corresponding densities. The proofs combine an effective joint Sato--Tate theorem with a geometric analysis of the associated semi-algebraic regions. As applications, we obtain new multiplicity one criteria, improve a theorem of Matom\"aki on small differences between Fourier coefficients, establish density-one analogues in the spirit of the Atkin--Serre conjecture, and derive a new characterization of twist-equivalence through the distribution of ratios of Fourier coefficients.
 \end{abstract}
\maketitle
\section{Introduction and statement of results}
Let $k\ge2$ be an even integer and $N\ge1$ an integer. We denote by $S_k(N)$ the space of cusp forms of weight $k$ for the congruence subgroup $\Gamma_0(N)$. Throughout the paper,
$f\in S_k(N)$ and $f'\in S_{k'}(N')$ are normalized non-CM newforms, with Fourier expansions 
$$
        f(z) = \sum_{n\ge 1} a(n)n^{(k-1)/2}e^{ 2 \pi i n z} \quad\quad {\rm and}\quad \quad f'(z) = \sum_{n\ge 1} a'(n)n^{(k'-1)/2}e^{ 2 \pi i n z}.
$$
We say that $f$ and $f'$ are \emph{twist-equivalent}, if there exists a Dirichlet character $\chi$ such that
$
f'=f\otimes\chi,
$
where
$
(f\otimes\chi)(z)
=\sum_{n\ge1}\chi(n)a(n)n^{(k-1)/2}e^{2\pi inz}.
$

One of the central problems in the theory of modular forms is determining newforms from partial information about their Fourier coefficients, or more generally, from arithmetic relations satisfied by them.
The first decisive result in this direction is the multiplicity one theorem, which asserts that 
if 
$$a(p) = a'(p)$$
for all but finitely many primes $p$ then $f=f'.$
Over the years, considerable effort has been devoted to strengthening and generalizing this principle (see \cites{duke-kowalski,
rajan21,  wang-cheng} and references therein). A major advance is the strong multiplicity one theorem (\cites{murty-pujahari, patankar-rajan}), which shows that if $a(p)=a'(p)$ on a set of primes $p$ of positive density, then  $f$ and $f'$ are twist-equivalent. 
In other words, exact agreement on a sufficiently large set of primes forces a strong relationship between the underlying forms.
More recently, \cite{ak}*{Theorem 2.17} and \cite{kkm2}*{Corollary 2.2} proved that if
$$
P(a(p^m),a'(p^n))=0,
$$
for some non-trivial polynomial $P$ and positive integers $m$ and $n$ on a set of primes of positive density, then $f$ and $f'$ are twist-equivalent. These results show that exact equality of Hecke eigenvalues is only one instance of a more general rigidity phenomenon concerning the Fourier coefficients of newforms.

The results mentioned above concern exact equalities or polynomial identities among Fourier coefficients. A natural next step is to investigate whether similar rigidity persists when the Fourier coefficients are close to one another rather than being identical. In this direction, Matom\"aki \cite[Theorem~3]{matomaki1} proved that if
$$
|a(p)-a'(p)|\le \frac1{50}
$$
for a set of primes of analytic density exceeding $15/31$, then $f$ and $f'$ are twist-equivalent. This shows that exact equality of Fourier coefficients may be replaced by a sufficiently strong closeness condition.

The above result naturally raises the following question: to what extent can newforms still be distinguished when exact algebraic relations among their Fourier coefficients are replaced by arithmetic inequalities? The main purpose of this paper is to answer this question in several natural settings. More precisely, we establish asymptotic formulas for the number of primes $p$ satisfying inequalities of the form
$$
|a(p^m)-a'(p^n)|<\varepsilon,\qquad
|a(p^m)a'(p^n)|<\varepsilon,\qquad
\left|
\frac{a(p^m)}{a'(p^n)}-t
\right|<\varepsilon,
$$
where $m,n$ are positive integers, $\varepsilon>0$, and $t\in\mathbf{R}$. These asymptotic formulas have several arithmetic applications.  We obtain new multiplicity one criteria, improve Matom\"aki's density theorem ({\em loc. cit.}), establish effective bounds for the least prime satisfying prescribed inequalities, derive density-one analogues in the spirit of the Atkin--Serre conjecture, and prove that the ratios of Fourier coefficients are quantitatively dense in $\R$.

Although these results concern different arithmetic relations, we prove them using a common approach which is based on an effective joint Sato--Tate theorem established in our earlier work \cite{kkm2}. The coefficients at prime powers arise naturally in symmetric power $L$-functions and satisfy
$ 
a(p^m)=U_m\!\left(\frac{a(p)}2\right),
$ for $p\nmid N$,
where $U_m$ denotes the Chebyshev polynomial of degree $m$ of the second kind. Therefore, the inequalities considered above are transformed into questions concerning the joint distribution of pairs
$$
(a(p),a'(p))
$$
inside suitable semi-algebraic subsets of $[-2,2]^2$. The required asymptotic formulas then follow from the effective joint Sato--Tate theorem established in \cite{kkm2} together with a geometric analysis of these regions. 

\subsection*{Results}
We begin with differences of Fourier coefficients $|a(p^m)-a'(p^n)|$. 
The following theorem gives an asymptotic formula for the number of primes at which these differences are smaller than a prescribed positive constant.
\begin{theorem}\label{thm:sympower-diagonal1}
Let $f \in S_k(N)$ and $f'\in S_{k'}(N')$ be twist-inequivalent non-CM newforms.
Then for any $\varepsilon>0$, and positive integers $m,n$, we have
 $${\#\left\{p\le x:
\left|a(p^m)-a'(p^n)\right|< \varepsilon
\right\}
=
C_{m,n,\varepsilon} \pi(x) + O\!\left(
\frac{\pi(x) (\log(kk'NN' \log x))^{1/4}}{(\log x)^{1/8}}
\right)},$$
where the constant $C_{m,n,\e}$ is positive and depends on $m,n$, and $\e$, whereas, the implied constant depends only on $m$ and $n$. Moreover, if $\e<1$ and  $m\le n$, then 
$$ 
\frac{(2n-1)}{144 \pi (m+2)(n+1)} \e\le
C_{m,n,\varepsilon}  \le \frac{32}{\pi^2}  m  (2\varepsilon)^{1/m}.
$$
\end{theorem}
The constant $C_{m,n,\varepsilon}$ admits a natural geometric interpretation, and it is precisely the joint Sato--Tate measure of the semi-algebraic region
$$
\left\{
(u,v)\in[-2,2]^2:
\left|U_m\!\left(\frac u2\right)
-
U_n\!\left(\frac v2\right)\right|
<\varepsilon
\right\}.
$$
The explicit upper bound for $C_{m,n,\varepsilon}$  yields the following multiplicity one type criterion.
\begin{corollary}\label{cor:difference}
Let  $f$ and $f'$ be non-CM newforms,  $\varepsilon<1$ be a positive real number, and $m\le n$ be positive integers.
If the set of primes $p$ such that
    $$
|a(p^m)-a'(p^n)|<\varepsilon 
    $$
    is of density greater than $\frac{32}{\pi^2} m  (2\varepsilon)^{1/m}$, then $f$ and $f'$ are twist-equivalent.
\end{corollary}
Taking $\varepsilon=1/50$ and $m=n=1$, we obtain that if the set $ \{p:\ |a(p)-a'(p)|<1/50\}$ has density exceeding $13/100$, then $f$ and $f'$ are twist-equivalent. Compared with \cite[Theorem 3]{matomaki1}, our result lowers the density bound from $15/31$ to $12/100,$ and moreover extends the result to coefficients at arbitrary prime powers.

The asymptotic formula obtained in Theorem \ref{thm:sympower-diagonal1} also lead to effective results concerning the least prime at which the Fourier coefficients are either very close or sparse.  Since both the main term and the error term in Theorem \ref{thm:sympower-diagonal1} are explicit, comparing them yields the following effective least-prime result.
\begin{corollary}\label{cor: least prime results}
Let $f$ and $f'$ be twist-inequivalent non-CM newforms, and  $\varepsilon<1$. Then there exist effectively computable constants $C_i=C_i(m,n,\varepsilon)>0$ and primes $p_i$, $i=1,2$ such that
$$
p_i<(kk'NN')^{C_i\log(kk'NN')}\quad \text{satisfy}\quad |a(p_1^m)-a'(p_1^n)|<\varepsilon \quad \text{and} \quad |a(p_2^m)-a'(p_2^n)|>\varepsilon.
$$
\end{corollary}
The constants $C_1$ and $C_2$ appearing above can be made explicit using 
bounds of $C_{m,n,\varepsilon}$ given in Theorem \ref{thm:sympower-diagonal1}  for small $\e.$
 We also remark that the large difference result $|a(p^m)-a'(p^n)|>\varepsilon$ should be compared with \cite[Proposition 1]{munshi10}, where related questions are studied using Rankin--Selberg theory in the context of determining modular forms through central values of $L$-functions.

By the symmetry of the joint Sato--Tate measure, the same conclusions remain valid if the difference
$
a(p^m)-a'(p^n)
$
is replaced by the sum
$
a(p^m)+a'(p^n),
$
and will not be stated separately.

We next study the primes for which the product $|a(p^m)a'(p^n)|$ is smaller than a fixes positive constant, which also detects the primes for which at least one of the values $|a(p^m)|$ or $|a'(p^n)|$ is small.  Related questions on the occurrence of small Fourier coefficients of a newform have been investigated extensively in the literature (see, for example, \cite{atkser,gaf}). We have the following result in which, as before, the main term is given by the joint Sato--Tate measure of a naturally associated semi-algebraic region, and we obtain explicit upper and lower bounds for the corresponding density.
\begin{theorem}\label{thm: sympower product}
    Let $f \in S_k(N)$ and $f'\in S_{k'}(N')$ be twist-inequivalent non-CM newforms.
Then for any $\e>0$, and positive integers $m,n$, we have
 $${\#\left\{p\le x:
\left|a(p^m)a'(p^n)\right|< \varepsilon
\right\}
=  
D_{m,n,\varepsilon} \pi(x) + O\!\left(
\frac{\pi(x) (\log(kk'NN' \log x))^{1/4}}{(\log x)^{1/8}}
\right)}.$$ 
Moreover, if $\e<1$, then 
$$
\frac{64mn}{9\pi^4}\left(1- \frac{\sqrt{\e}}{m+1}\right) \left(1- \frac{\sqrt{\e}}{n+1}\right) \left(\frac{\e}{(m+1)(n+1)}\right) \le 
D_{m,n,\varepsilon}  \le \frac{32}{\pi^2} \left(m \e^{1/(2m)}+n \e^{1/(2n)} \right). 
$$
\end{theorem}
Note that Theorem~\ref{thm: sympower product} also yields analogues of Corollaries~\ref{cor:difference} and~\ref{cor: least prime results}, with the difference of Fourier coefficients replaced by their product. Indeed, the explicit upper bound for $D_{m,n,\varepsilon}$ gives a multiplicity one criterion for products of Fourier coefficients analogous to Corollary~\ref{cor:difference}. Moreover, applying the same argument as in the proof of Corollary~\ref{cor: least prime results} together with Theorem~\ref{thm: sympower product}, one obtains explicit bounds for the least prime $p$ satisfying either of the inequalities
$$
|a(p^m)a'(p^n)|<\varepsilon,
\qquad
|a(p^m)a'(p^n)|>\varepsilon.
$$

Theorems~\ref{thm:sympower-diagonal1} and~\ref{thm: sympower product}
show that the prescribed inequalities hold for a positive density of primes when  $\varepsilon$ is fixed. It is natural to ask what happens when $\varepsilon$ is allowed to depend on the prime $p$ and tends to zero as $p\to\infty$. This question is closely related to the philosophy underlying the Atkin–Serre conjecture \cite{atkser}. Let $f\in S_k(N)$ be a non-CM newform of weight $k\ge 4$. The Atkin--Serre conjecture predicts that, for every $\varepsilon>0$, there exists a constant $c_{\varepsilon,f}>0$ such that
$$
|a(p)|\ge c_{\varepsilon,f}p^{-(1+\varepsilon)}
$$
for all sufficiently large primes $p$. 
Although the conjecture remains open, Gafni, Thorner and Wong~\cite[Theorem~1.1]{gaf} established a remarkable density-one analogue by proving that for a set of primes of density one
$$
|a(p)|
\ge
\frac{2\log\log p}
{\sqrt{\log p}}.
$$
In the same spirit, we next study the distribution of primes for which the size of the differences and products of the Fourier coefficients of two newforms decreases as the prime varies. The following theorem shows that if the rate of decrease is
$$
\frac{(\log\log p)^{1/4}}
{(\log p)^{1/8}},
$$
then the corresponding inequalities are satisfied by only a density-zero set of primes.
\begin{theorem}\label{Atkin}
    Let $f\in S_k(N)$ and $f'\in S_k'(N')$ be twist-inequivalent non-CM newforms, and $m\le n$ be positive integers.
Then as $x\to\infty$
\begin{enumerate}
    \item 
$\displaystyle{\!
\# \!\left\{x< p\le 2 x : 
|a(p^m)-a'(p^n)|\!<\! \frac{4(\log  (\log p))^{1/4}}{(\log p)^{1/8}}
\right\}
\!=\!
O\!\left(
\frac{\pi(x)  (\log(kk'NN' \log x))^{1/4m}}{(\log x)^{1/8m}}
\right)}$, \\
\item 
$\displaystyle{
    \#\left\{x< p\le 2 x:
|a(p^m)a'(p^n)| < \frac{4(\log  (\log p))^{1/4}}{(\log p)^{1/8}}
\right\}
\!=\!
O\!\left(
\frac{\pi(x)  (\log(kk'NN' \log x))^{1/8n}}{(\log x)^{1/16n}}
\right)}.
$ 
\end{enumerate}
 The implied constants in both cases depend only on $m$ and $n$.  
\end{theorem}
As an immediate consequence, we obtain a density-one statement which may be viewed as a two form analogue of the theorem of Gafni, Thorner and Wong \cite[Theorem 1.1]{gaf}. 
\begin{corollary}\label{Atkin-Serre}
    Let $f$ and $f'$ be twist-inequivalent non-CM newforms and $m, n$ be positive integers. Then  for a set of primes $p$ of density one, we have 
$$ |a(p^m)-a'(p^n)|\ge \frac{4(\log (\log p))^{1/4}}{(\log p)^{1/8}} \quad and \quad 
|a(p^m)a'(p^n)|\ge \frac{4(\log  (\log p))^{1/4}}{(\log p)^{1/8}}.$$
\end{corollary}

Our final result concerns the global distribution of the ratios of Fourier coefficients $
\frac{a(p^m)}{a'(p^n)}, 
$ whenever $a'(p^n)\neq 0$. Unlike differences and products, ratios measure the relative behavior of the coefficients and lead to a remarkably rich distribution theorem. 
The following theorem shows that these ratios are distributed throughout the real line in a quantitative sense.
\begin{theorem}\label{thm:density-ratio}
Let $f\in S_k(N)$ and $f'\in S_k'(N')$ be twist-inequivalent non-CM newforms, and $m,n$ be positive integers. Then for every $t\in \R$ and $\varepsilon>0$, there exists a positive, computable constant $ F_{m,n,\e}^{t}$ such that
$$
\#\left\{
p\le x :
a'(p^n)\neq 0,\,
\left|\frac{a(p^m)}{a'(p^n)}-t\right|<\varepsilon
\right\} =  F_{m,n,\e}^{t} \pi(x)+ O\!\left(
\frac{\pi(x) (\log(kk'NN' \log x))^{1/4}}{(\log x)^{1/8}}
\right).
$$
\end{theorem}
It follows from the above theorem that every non-empty open interval of $\R$ contains ratios of Fourier coefficients for a positive density of primes.
As an application, we obtain a new multiplicity one criterion expressed purely in terms of the topological properties of the ratio set.
\begin{corollary}
Let $f$ and $f'$ be non-CM newforms, and let $m,n$ be positive integers. Let $\mathcal P$ be a set of primes of density 1. If the closure of the set
$$
\left\{
\frac{a(p^m)}{a'(p^n)}
:
p \in \mathcal P,\ a'(p^n)\neq 0
\right\}
$$
is a proper subset of $\R$, then $f$ and $f'$ are twist-equivalent.
\end{corollary} 
Hence, for twist-inequivalent newforms, the set 
$\{\frac{a(p^m)}{a'(p^n)}: p~prime, a'(p^n)\neq 0\}$ is dense in $\R$.

\textbf{Outline of the paper.}
The paper is organized as follows. In Section \ref{sec:preliminaries}, we recall an effective joint Sato--Tate theorem and a quantitative estimate of Kleinbock--Margulis \cite{margulis} for the measure of the set on which a polynomial assumes small values. Sections \ref{sec: proof of sympower diagonal} and \ref{sec: proof of sympower product} contain the proofs of Theorems \ref{thm:sympower-diagonal1} and \ref{thm: sympower product}. Section \ref{sec: proof of atkin} is devoted to the proof of the Atkin--Serre type result, while Section \ref{sec: proof of least prime} establishes the least-prime results. We present the proof of Theorem \ref{thm:density-ratio} in Section \ref{sec: proof of density in R}. The final section contains the concluding remarks. 

\textbf{Notation.}
Throughout this article, $f$ and $f'$ denote two normalized non-CM newforms, $p$ denotes a prime number, and  $\ell, m, n$  denote positive integers.
For any set $S$, we write $\#S$ for its cardinality, and $\mathbbm{1}_S$ to denote its indicator function. Let $\pi(x) = \#\{p: p\le x\}$. For any set $\mathcal P$ consisting of primes, we define the (natural) density of $\mathcal{P}$ as $\lim_ {x\to \infty}\frac{\#(\mathcal{P}\cap \{p\le x\})}{\pi(x)}$, whenever the limit exists.
Given a function $F$ and a positive function $G$, the notation $F = O(G)$ (or equivalently $F \ll G$) means that there exists a positive constant $c$ such that $|F| \le c\,G$ in the relevant range. For a set $E$, we denote its boundary by $\partial E$. We call a subset $E$ of $\R^2$ a semi-algebraic region if it can be written as the intersection of finitely many polynomial inequalities.

 \setcounter{theorem}{0}
\numberwithin{theorem}{section}
\section{ Preliminaries}\label{sec:preliminaries}
In this section, we collect some results that will be used throughout the paper.
\subsection{Effective joint Sato--Tate distribution and important results}
Let $$f(z)=\sum_{n\ge 1}a(n)n^{\frac{k-1}{2}}e^{2 \pi i n  z} \in S_{k}(N)\quad {\text{and}}\quad f'(z)=\sum_{n\ge 1}a'(n)n^{\frac{k'-1}{2}}e^{2 \pi i nz}\in S_{k'}(N')$$
be two normalized non-CM newforms that are twist-inequivalent. By joint Sato--Tate theorem  \cite{wong}, the pair sequence $(a(p),a'(p))$ is equidistributed in $[-2,2]^2$ with respect to the joint Sato--Tate measure
$$d\mu_{\rm JST}=\frac{1}{\pi^2}\sqrt{1-\frac{u^2}{4}}\sqrt{1-\frac{v^2}{4}}\,dudv.$$ 
More precisely, if $E\subset [-2,2]^2$ is a measurable set whose boundary has the Lebesgue measure zero, then
$$
        \#\{ p \le x : (a(p),a'(p))\in E  \} \sim \mu_{\rm JST}(E) \pi(x), \quad {\text{as } } x\to\infty.
$$

Recently, the authors of this article, building on earlier works of Thorner \cite{thorner2} and Chen--Shen \cite{Shen}, established the following effective joint Sato--Tate theorem for any measurable subset $E$ of $[-2,2]^2$ whose boundary is a finite union of continuous curves of finite length. 
 \begin{theorem}[{\cite[Theorem 1.3]{kkm2}}]\label{thm:main1}
  Let $f\in S_k(N)$ and $f'\in S_k'(N')$ be  twist-inequivalent non-CM newforms, 
 and  $E \subseteq [-2, 2]^2$ be a Borel measurable set such that 
 $$
    \partial E = \gamma_1 \cup \gamma_2 \cup\cdots \cup \gamma_\alpha,
 $$ 
 with $ L \coloneqq {\rm length}(\partial E)$, and each  $\gamma_i$ is either a vertical line or a continuous curve with $$\#\{ \gamma_i \cap \{ (u,v)\in [-2,2]^2 : u=a\} \} \le \beta, \quad \text{for some  fixed } \beta>0.$$
 Then
\begin{equation*}\label{eq:main2}
     \#\{p\le x: (a(p),a'(p))\in E\}=\mu_{\rm JST}(E)\pi(x)+ O\!\left(
\frac{ L \alpha \beta    \pi(x) (\log(kk'NN' \log x))^{1/4}}{(\log x)^{1/8}}
\right).
    \end{equation*}
    Here, the implied constant is positive, absolute, and effectively computable. 
 \end{theorem}
They further derive a useful special case of this theorem for semi-algebraic regions inside $[-2,2]^2$. We now state this refined version, which will play a key role in proving our results.
\begin{theorem}[{\cite[Theorem 2.1]{kkm2}}]\label{thm:main2}
Let $f$ and $f'$ be two twist-inequivalent newforms. Let $P(u,v)\in \R[u,v]$ be a non-constant polynomial of degree $\delta$. For any real interval $I$   we have
$$
\#\{p\le x :  P(a(p),a'(p))\in I\}= \mu_{\rm JST}(E_{I,P})\pi(x) + O\left(
\frac{\delta^8\pi(x) (\log(kk'NN' \log x))^{1/4}}{(\log x)^{1/8}}
\right),$$ 
where 
\begin{equation}
    E_{I,P}:=\{(u,v)\in [-2,2]^2: P(u,v)\in I\}
\end{equation}
and the implied constant is positive, absolute, and effectively computable.
\end{theorem}

\subsection{A Measure Approximation Result}
We recall a quantitative estimate that provides an upper bound for the measure of the set on which a polynomial takes small values. Such results will be useful in controlling contributions from regions where certain polynomial expressions are small.

For a polynomial $P(u)\in \R[u]$ and an interval $I$, define
$$
    \lVert P \rVert_I = \sup_{u\in I} |P(u)|.
$$
For a set $A$, let $\mathbbm{1}_A$ denote its indicator function.
Then, we have the following result due to Kleinbock and Margulis. 
\begin{proposition}[{{\cite[Proposition 3.2]{margulis}}}]\label{thm: margulis}
Let $P(u)\in \R[u]$ be a polynomial of degree $\ell$, and let $I=[a,b]$. Then for any $\e>0$,
$$
    \int_I \mathbbm{1}_{\{u : |P(u)| < \e\}} \, du
    \le 2 (b-a)\,\ell\,(\ell+1)^{1/\ell}
    \left(\frac{\e}{\lVert P \rVert_I}\right)^{1/\ell}.
$$
\end{proposition}
We remark that the original result \cite[Proposition 3.2]{margulis} is stated for an open interval $I$. However, because a closed interval $[a,b]$ can be written as union of open intervals $(a-\frac{1}{j}, b+ \frac{1}{j})$, $j\ge 1$,  thus using the limiting argument gives Proposition \ref{thm: margulis}.

\section{Proof of Theorem \ref{thm:sympower-diagonal1}}\label{sec: proof of sympower diagonal}
For primes $p\nmid NN'$, the Hecke relations imply that the Fourier coefficients at prime powers satisfy 
\begin{equation}\label{eq:hecke relation}
     a(p^m)=U_m\left(\frac{a(p)}{2}\right) \quad {\rm and } \quad  a'(p^n)= U_n\left(\frac{a'(p)}{2}\right),
\end{equation}
where 
$U_m, U_n$ denote the Chebyshev polynomial of the second kind of degree $m$ and $n$, respectively.
Define the polynomial $P(u,v) = U_m(u/2)-U_n(v/2)$ and consider the set
\begin{equation}\label{eq: E m n epsilon}
      E_{m,n,\varepsilon}:= \{(u,v)\in [-2,2]^2: -\e < P(u,v)< \e \}.
\end{equation}
Then for every prime $p\nmid NN'$, we have the equivalence
$$|a(p^m)-a'(p^n)| < \varepsilon \quad\text{if and only if}\quad (a(p),a'(p))\in E_{m,n,\e}.$$
Thus, 
$$\#\{p\le x:|a(p^m)-a'(p^n)|< \e\}
                    =  \#\{p\le x:(a(p),a'(p))\in E_{m,n,\e}\}+O(1),$$
    where the $O(1)$ term accounts for the finitely many primes dividing $NN'$.
  Applying Theorem \ref{thm:main2} for the polynomial $P(u,v)$ and the interval $(-\e,\e)$, we obtain
\begin{equation*}\label{eq:Setb1}
            \#\{p\le x:|a(p^m)-a'(p^n)|< \e\}
                    = 
                \mu_{\rm JST}(E_{m,n,\e})\pi(x)
                +
               O\!\left(
\frac{\pi(x) (\log(kk'NN' \log x))^{1/4}}{(\log x)^{1/8}}
\right),
\end{equation*}
with the implied constant depending on $m$ and $n.$
 Setting $$C_{m,n,\e}:=\mu_{\rm JST}(E_{m,n,\e}),$$ we obtain the required asymptotic formula.

\textbf{Lower bound of $\boldsymbol{C_{m,n,\e}}$:}
Using the parametrization
$$u=2 \cos \theta,\, v= 2\cos\phi, \quad \theta, \phi \in[0,\pi]$$
and the identity
$$U_{\ell}(\cos \psi)=\frac{\sin ((\ell+1)\psi)}{\sin \psi}=:F_{\ell}(\psi),$$
the joint Sato--Tate measure becomes
\begin{equation}\label{eq: C m n e polar}
        C_{m,n,\e} = \frac{4}{\pi^2}\int_{(\theta,\phi)\in [0,\pi]^2} \mathbbm{1}_{\{(\theta,\phi): |F_m(\theta)-F_n(\phi)|< \e\}}\, \sin^2\theta \sin^2\phi\, d\theta d\phi,
\end{equation}
where for a set $A$, $\mathbbm{1}_{A}$ denotes the indicator function of the set $A$. Restricting the domain of integration to $\theta,\phi \in  [\frac{\pi}{6}, \frac{5\pi}{6}]$, yields
$$ C_{m,n,\e} \geq \frac{1}{4\pi^2}\int_{(\theta,\phi)\in [\frac{\pi}{6},\frac{5\pi}{6}]^2} \mathbbm{1}_{\{(\theta,\phi): |F_m(\theta)-F_n(\phi)|< \e\}}\, d\theta d\phi,
$$
because the integrand is non-negative and on this interval $\sin \theta, \sin \phi \geq \frac{1}{2}.$
We now further restrict the integral to $\phi \in [\frac{\pi}{6}, \frac{5\pi}{6}]$ satisfying $|\sin ((n+1)\phi)|\le \frac{{1}}{2}$ to obtain
\begin{equation}\label{eq: inner integral difference lower bound}
        C_{m,n,\e} \ge \frac{1}{ 4\pi^2} \int_{\frac{\pi}{6}}^{\frac{5\pi}{6}} \left(  \int_{\frac{\pi}{6}}^{\frac{5\pi}{6}} \mathbbm{1}_{\{\theta: |F_m(\theta)-F_n(\phi)|< \e\}}\, d \theta \right)  \left(\mathbbm{1}_{\{\phi: |\sin((n+1)\phi)| \le \frac{1}{2}\}}\right)\, d\phi.
\end{equation}
Fix $\phi$ such that $|\sin(n+1)\phi)| \le \frac{1}{2}$. Then using $F_n(\phi)=\frac{\sin ((n+1)\phi)}{\sin \phi}$, and the fact that $\sin \phi \geq \frac{1}{2}$ on the interval $[\frac{\pi}{6},\frac{5\pi}{6}]$, we have $$F_n(\phi)\in [-1,1].$$
Since the function $F_m$ is continuous on $[\frac{\pi}{6}, \frac{5\pi}{6}]$  and its range contains $[-1,1]$,  therefore for each $\phi$ in the  domain of integration, there exists a point $\theta_\phi \in [\frac{\pi}{6}, \frac{5\pi}{6}]$ such that
$$F_m(\theta_\phi)=F_n(\phi).$$
Hence, for such  a $\phi$ the inner integral of \eqref{eq: inner integral difference lower bound} becomes
\begin{equation*}
    \int_{\frac{\pi}{6}}^{\frac{5\pi}{6}} \mathbbm{1}_{\{\theta : |F_m(\theta)-F_n(\phi)|< \e\}} \,d \theta =  \int_{\frac{\pi}{6}}^{\frac{5\pi}{6}} \mathbbm{1}_{\{\theta:|F_m(\theta)-F_m(\theta_\phi)|< \e\}}\, d \theta.
\end{equation*}
Now, as $F_m$ is differentiable, a direct derivative computation gives
$$F'_m(\theta)\le 4(m+2), \quad {{\text{ for all }}} \theta \in \bigg[\frac{\pi}{6}, \frac{5\pi}{6}\bigg]$$ and hence $F_m$ is Lipschitz on $[\frac{\pi}{6}, \frac{5\pi}{6}]$. It follows that
$$|F_m(\theta)-F_m(\theta_\phi)|< \e, \quad \text{ whenever } |\theta-\theta_\phi| < \frac{\e}{4(m+2)}.$$
Thus,
$$
       \int_{\frac{\pi}{6}}^{\frac{5\pi}{6}} \mathbbm{1}_{\{\theta : |F_m(\theta)-F_m(\theta_\phi)|< \e\}} \,d \theta \ge  \int_{\frac{\pi}{6}}^{\frac{5\pi}{6}} \mathbbm{1}_{\{\theta : |\theta-\theta_\phi|< \frac{\e}{4(m+2)}\}}\, d \theta \ge \frac{\e}{4(m+2)}.   
$$
Substituting this bound into inequality \eqref{eq: inner integral difference lower bound}, we obtain
\begin{equation}\label{eq: outer integral lower bound}
        C_{m,n,\e} \ge \frac{\e}{16 (m+2)\pi^2}  \int_{\frac{\pi}{6}}^{\frac{5\pi}{6}} \mathbbm{1}_{\{\phi : |\sin((n+1)\phi)| \le \frac{1}{2}\}} d\phi.     
\end{equation}
It remains to estimate the measure of the set
$$\left\{\phi \in \left[\frac{\pi}{6}, \frac{5\pi}{6}\right]:|\sin((n+1)\phi)|\le \frac{1}{2}\right\}.$$

Observe that
$$|\sin((n+1)\phi)|\le \frac{1}{2} \iff
\phi\in \bigcup\limits_{j\in \Z} \left[\left(j\pi-\frac{\pi}{6}\right) \frac{1}{n+1},\left(j\pi+\frac{\pi}{6}\right) \frac{1}{n+1}\right].
$$
Each of these intervals has length $\frac{\pi}{3(n+1)}$ and a simple counting argument shows that the interval $[\frac{\pi}{6},\frac{5\pi}{6}]$ contains at least $\lfloor \frac{2n+1}{3} \rfloor$ such disjoint subintervals.
Therefore, $$
        \int_{\frac{\pi}{6}}^{\frac{5\pi}{6}} \mathbbm{1}_{\{\phi: |\sin((n+1)\phi)| \le \frac{1}{2}\}} d\phi \ge  \left\lfloor\frac{2n+1}{3}\right \rfloor \frac{\pi}{3(n+1)}.
$$
Substituting this estimate into \eqref{eq: outer integral lower bound} and using the inequality $\left\lfloor\frac{2n+1}{3}\right \rfloor \ge \frac{n-1}{3}$, 
we obtain for $m, n \ge 1$
$$
        C_{m,n,\e} \ge  \frac{(2n-1)}{144 \pi(m+2)(n+1)}\e,
$$
and this completes the proof.

\textbf{Upper bound of $\boldsymbol{C_{m,n,\e}}$:}
By discarding the weight functions, we obtain
\begin{equation}\label{eq: C m n epsilon}
        C_{m,n,\e} 
        \le  \frac{1}{\pi^2} \int_{-2}^{2}\left( \int_{-2}^{2} \mathbbm{1}_{\{u:|U_m(u/2)-U_n(v/2)|< \e\}} \,du\right)dv.
\end{equation}
 We now estimate the inner integral. For a fixed $v\in [-2,2]$, applying Proposition \ref{thm: margulis} to the polynomial $U_m(u/2)-U_n(v/2)$ in the variable $u$ for interval $I=[-2,2]$, we obtain
\begin{equation}\label{eq: margulis result}
        \int_{-2}^{2}  \mathbbm{1}_{\{u : |U_m(u/2) - U_n(v/2)|< \e\}}du  \le 8 m (m+1)^{1/m}\left(\frac{\e}{\lVert U_m(u/2) - U_n(v/2) \rVert_{I}}\right)^{1/m}.
  \end{equation}
To bound the denominator from below, we note that for a fixed real number $A$ and a polynomial $P(u)$ defined on $I$, we have
$$
         \lVert P(u) - A \rVert_{I} \ge \frac{1}{2} \left(\max_{u\in I} (P(u)) - \min_{u\in I}(P(u)) \right).
$$
Putting $P(u)=U_m(u/2)$, and using  $\min_{u\in I}(U_m(u/2)) \le -1$, $\max_{u\in I}(U_m(u/2)) = m+1$,  we have
\begin{equation*}\label{eq: lower bound for sup norm}
    \lVert U_m(u/2) - U_n(v/2) \rVert_{I} \ge \frac{m+2}{2},
\end{equation*}
Using this estimate in inequality \eqref{eq: margulis result}, we obtain
\begin{equation*}\label{eq: difference sup norm}
        \int_{-2}^{2}  \mathbbm{1}_{\{|U_m(u/2)-U_n(v/2)| < \e\}}du \le 8  m (2\e)^{1/m}.
\end{equation*}
Substituting it into \eqref{eq: C m n epsilon}, we conclude that
$$
            C_{m,n,\e} \le \frac{32  }{\pi^2}  m (2\e)^{1/m},
$$
completing the proof.

\section{Proof of Theorem \ref{thm: sympower product}} \label{sec: proof of sympower product}
Let $P(u,v) = U_m(u/2)U_n(v/2)$ and 
$$
        E_{m,n,\e}^{'} \coloneqq \{ (u,v)\in [-2,2]^2 : -\e < P(u,v) < \e\}.
$$

Then, for $p\nmid NN'$,  using the Hecke relations, we have 
$$
    |a(p^m)a'(p^n)| < \varepsilon \quad\text{if and only if}\quad (a(p),a'(p))\in E_{m,n,\e}^{'}.
$$
Hence, by Theorem \ref{thm:main2},
 \begin{equation}\label{eq: Setproduct}
            \#\{p\le x:|a(p^m)a'(p^n)|< \e\} = 
                \mu_{\rm JST}(E_{m,n,\e}^{'})\pi(x)
                +
               O\!\left(
\frac{\pi(x) (\log(kk'NN' \log x))^{1/4}}{(\log x)^{1/8}}
\right).
\end{equation}
Setting $D_{m,n,\e}=\mu_{\rm JST}(E_{m,n,\e}^{'})$, proves the asymptotic formula.
We now establish lower and upper bounds for $D_{m,n,\e}.$

\textbf{Lower bound of $\boldsymbol{D_{m,n,\e}}$:}
Observe that
$$
       E_{m,n,\e}^{'} \supset \{ u\in [-2,2] : |U_m(u/2)|< \sqrt{\e}\} \times \{ v\in [-2,2] : |U_n(v/2)|< \sqrt{\e}\}.
$$
Making the change of variables
$$u=2\cos \theta, \quad v = 2\cos\phi, \quad \text{ with }  \theta, \phi \in [0,\pi],$$
we obtain
\begin{equation}\label{eq: D m n e lower bound}
        D_{m,n,\e} \ge \frac{4}{\pi^2} \left(\int_{0}^{\pi} \mathbbm{1}_{\{ \theta : |U_m(\cos\theta)|< \sqrt{\e} \}} \sin^2 \theta \, d\theta\right) \left(\int_{0}^{\pi} \mathbbm{1}_{\{ \phi : |U_n(\cos\phi)|< \sqrt{\e} \}} \sin^2 \phi \, d\phi\right).
\end{equation}
It is therefore sufficient to give a lower bound for the integral
$$
        \int_{0}^{\pi} \mathbbm{1}_{\{ \psi : |U_{\ell}(\cos\psi)|< \sqrt{\e} \}} \sin^2 \psi \, d\psi,
$$
where $\ell\ge 1$ is an integer. For that, we restrict the domain of integration to suitable intervals around the zeros of $U_\ell(\cos \psi)$ lying in the interval $[0, \pi]$.

Note that $U_\ell(\cos\psi)$ vanishes at $\frac{j \pi}{\ell+1}$, for $1\le j\le \ell.$ Thus, for $1\le j\le \ell$,  set
\begin{equation}\label{eq: deltaj}
    \psi_j =   \frac{j \pi}{\ell+1}, \quad  \delta_j = \frac{\sqrt{\e} \sin \psi_j}{ 2(\ell+1)}, \quad{\rm and}\quad I_j = (\psi_j - \delta_j, \psi_j + \delta_j). 
\end{equation}
The following assertions hold for any positive integer $1\le j\le \ell$.
\begin{enumerate}
    \item $I_j$'s are mutually disjoint subintervals of $[0,\pi]$.
    \item $\sin\psi > \left(1-\frac{\sqrt{\e}}{2(\ell+1)}\right)\sin \psi_j$ for all $\psi\in I_j.$
    \item $|U_\ell (\cos \psi)| < \sqrt{\e}$ for all $\psi\in I_j$.
\end{enumerate}
We now prove the above assertions.
\begin{proof}
$(i)$ To show that $I_j\subset [0,\pi]$, it is enough to show that $\psi_1-\delta_1> 0$ and $\psi_\ell+\delta_\ell<\pi$, which is immediate from the following observation
$$
        \delta_j = \frac{\sqrt{\e} \sin \psi_j}{ 2(\ell+1)} < \frac{\psi_j}{2(\ell+1)} < \frac{\pi}{2(\ell+1)}, \quad \text{for } j\ge 1,
$$
where we have used the facts that $\sin \psi\le \psi$ for $\psi>0$. Furthermore,  $\psi_{j+1}-\psi_j = \frac{\pi}{\ell+1}$, and hence $\psi_{j+1} + \delta_{j+1} > \psi_j + \delta_j$ for $1\le j\le \ell-1$, which shows that the intervals $I_j$'s are disjoint.

$(ii)$ Applying the mean value theorem to the sine function in the interval $[\psi, \psi_j]$ or $[\psi_j, \psi]$ for any $\psi\in I_j$, we have
$$
    |\sin\psi -\sin\psi_j| < \delta_j,
$$
equivalently, $-\delta_j < \sin\psi-\sin\psi_j < \delta_j$, thus the desired lower bound follows directly after substituting the value of $\delta_j$. 

$(iii)$ If $\psi\in I_j$, then $\psi = \psi_j + t$ with $|t|< \delta_j$, and $|\sin((\ell+1)\psi)|< (\ell+1)|\delta_j|$. Therefore, using part $(i)$ we have
$$
     |U_\ell(\cos \psi)| = \left\lvert \frac{\sin((\ell+1)\psi)}{\sin\psi}\right\rvert < \frac{\sqrt{\e} \sin \psi_j}{2 \sin \psi_j}< \sqrt{\e},     
$$
completing the proof.
\end{proof}

As a consequence, we obtain 
$$
          \int_{0}^{\pi} \mathbbm{1}_{\{ \psi : |U_{\ell}(\cos\psi)|< \sqrt{\e} \}} \sin^2 \psi \, d\psi \ge \sum_{1\le j \le \ell}  \left(1-\frac{\sqrt{\e}}{2(\ell+1)}\right)^2  \left(\sin^2\psi_j    \int_{I_j}  d\psi\right).
$$
Note that the interval $I_j$ has length $2\delta_j$, thus substituting the value of $\delta_j$ from \eqref{eq: deltaj} gives
\begin{equation}\label{eq: integral over I j}   
    \int_{0}^{\pi} \mathbbm{1}_{\{ \psi : |U_{\ell}(\cos\psi)|< \sqrt{\e} \}} \sin^2 \psi \, d\psi \ge \left(1-\frac{\sqrt{\e}}{2(\ell+1)}\right)^2 \frac{\sqrt{\e}}{\ell+1}    \sum_{1\le j\le \ell}    \sin^3\psi_j.
    \end{equation}
Using the trigonometric relation $\sin (3\psi) = 3 \sin \psi - 4 \sin^3 \psi$
and putting the value of $\psi_j$ from \eqref{eq: deltaj}, we have
$$
    \sum_{1\le j\le \ell}    \sin^3\psi_j = \frac{3}{4} \sum_{1\le j\le \ell}    \sin\left(j\frac{ \pi}{\ell+1}\right) - \frac{1}{4} \sum_{1\le j\le \ell} \sin \left(j\frac{3 \pi}{\ell+1}\right).    
$$
Applying the formula (see \cite[p. 37]{gradshteyn}) 
$$
\sum_{1\le j\le \ell} \sin (j\psi) = \frac{\sin \left(\frac{\ell \psi}{2}\right) \sin \left(\frac{(\ell+1) \psi}{2}\right)}{\sin \left(\frac{\psi}{2}\right)},
$$  
we obtain that
$$
      \sum_{1\le j\le \ell}    \sin^3\psi_j = \frac{3}{4}    \frac{\sin \left(\frac{\ell \pi}{2(\ell+1)}\right)}{\sin \left(\frac{ \pi}{2(\ell+1)}\right)} + \frac{1}{4} \frac{\sin \left(\frac{3\ell \pi}{2(\ell+1)}\right)}{\sin \left(\frac{ 3\pi}{2(\ell+1)}\right)} = \frac{3}{4} \cot \left(\frac{\pi}{2(\ell+1)}\right) - \frac{1}{4} \cot \left(\frac{3\pi}{2(\ell+1)}\right).
$$
Now, we invoke the relation $\cot(3\psi) = \frac{\cot^3 \psi - 3\cot \psi}{3\cot^2\psi -1}$ to obtain
$$
      \sum_{1\le j\le \ell}    \sin^3\psi_j = \frac{2\cot^3 \left(\frac{\pi}{2(\ell+1)}\right)}{ 3 \cot^2 \left(\frac{\pi}{2(\ell+1)}\right)-1 } \ge \frac{2}{3} \cot\left(\frac{\pi}{2(\ell+1)}\right) \ge \frac{2}{3} \left( \frac{2(\ell+1)}{\pi} - \frac{\pi}{4(\ell+1)} \right),    
$$
where we have used $\cot \psi > \frac{1}{\psi} - \frac{\psi}{2}$, for $\psi>0$. Therefore, we have
$$
       \sum_{1\le j\le \ell}    \sin^3\psi_j  \ge \frac{4\ell}{3\pi}. 
$$
Using this inequality in \eqref{eq: integral over I j}, we have
$$
        \int_{0}^{\pi} \mathbbm{1}_{\{ \psi : |U_{\ell}(\cos\psi)|< \sqrt{\e} \}} \sin^2 \psi \, d\psi \ge \frac{4\ell}{3\pi}\left(1-\frac{\sqrt{\e}}{2(\ell+1)}\right)^2 \frac{\sqrt{\e}}{\ell+1} \ge \frac{4\ell}{3\pi}\left(1-\frac{\sqrt{\e}}{\ell+1}\right)^2 \frac{\sqrt{\e}}{\ell+1}.  
$$
Substituting this in \eqref{eq: D m n e lower bound}, we obtain 
$$
        D_{m,n,\e} \ge \frac{64mn}{9\pi^4}\left(1- \frac{\sqrt{\e}}{m+1}\right) \left(1- \frac{\sqrt{\e}}{n+1}\right) \left(\frac{\e}{(m+1)(n+1)}\right).
$$

\textbf{Upper bound of $\boldsymbol{D_{m,n,\e}}$:}
It remains to estimate an upper bound of the constant $D_{m,n,\e}$. In order to achieve this, we start by noting that
$$
        E_{m,n,\e}^{'} \subset \{ (u,v)\in [-2,2]^2 : |U_m(u/2)|< \sqrt{\e}  \} \cup \{ (u,v)\in [-2,2]^2 : |U_n(v/2)|< \sqrt{\e}  \},
 $$
and therefore,
\begin{equation}\label{eq: D m n epsilon}
        D_{m,n,\e} 
        \le  \frac{1}{\pi^2} \int_{-2}^{2}\left( \int_{-2}^{2} \mathbbm{1}_{\{u : |U_m(u/2)|< \sqrt{\e}\}} du\right)dv +  \frac{1}{\pi^2}\int_{-2}^{2}\left( \int_{-2}^{2} \mathbbm{1}_{\{v : |U_n(v/2)|< \sqrt{\e}\}} dv\right)du.
\end{equation}
We apply Proposition \ref{thm: margulis} to  both the inner integrals to obtain
$$
        D_{m,n,\e} \le \frac{1}{\pi^2} \left(  8m \e^{1/(2m)} \int_{-2}^{2} dv  + 8n \e^{1/(2n)} \int_{-2}^{2} du  \right)=\frac{32}{\pi^2} \left(m \e^{1/(2m)}+n \e^{1/(2n)} \right),
$$
which gives the required upper bound.

\section{Proof of Theorem \ref{Atkin}}\label{sec: proof of atkin}
For simplicity, put
\begin{equation}\label{eq: defn M x}
 M(x) = \frac{4(\log( kk'NN' \log x))^{1/4}}{{( \log x)^{1/8}}}  \quad {\rm and} \quad  \e(x) = \frac{4(\log  (\log x))^{1/4}}{(\log x)^{1/8}}.
 \end{equation} 
Then for sufficiently large $x$, $M(x)$ is a strictly decreasing function, and therefore for any prime $p\in (x,2x]$, we have
 $$\e(p) = \frac{4(\log(\log p))^{1/4}}{{(\log p)^{1/8}}} < M(x).$$
It follows that
 \begin{equation}\label{eq: sympower difference estimation}
         \#\left\{p\in (x,2x]: |a(p^m)-a'(p^n)| < \e(p)  \right\} \le  \#\left\{p\in (x,2x]: |a(p^m)-a'(p^n)| < M(x)  \right\},
       \end{equation}
and 
\begin{equation}\label{eq: sympower product estimation}
         \#\left\{p\in (x,2x]: |a(p^m)a'(p^n)| < \e(p)  \right\} \le  \#\left\{p\in (x,2x]: |a(p^m)a'(p^n)| < M(x)  \right\}.
 \end{equation}
 Thus, it suffices to estimate these latter sets.

{\bf{Proof of $(i)$}}:  For a fixed $x$, define the set
$$
             E_{m,n,M(x)} = \{ (u,v) \in [-2,2]^2 : \lvert U_m(u/2)-U_n(v/2)\rvert < M(x)\}.
$$
For a prime $p\nmid NN'$, the Hecke relations give
$$
|a(p^m)-a'(p^n)| < M(x) \quad \text{if and only if} \quad (a(p),a'(p))\in E_{m,n,M(x)}.
$$
Without loss of generality, we may assume that $m\le n$. Therefore, applying Theorem \ref{thm:sympower-diagonal1} with $\e = M(x)$, we obtain
$$
        \#\left\{ p\in (x,2x]: |a(p^m)-a'(p^n)| < M(x) \right\} = C_{m,n,M(x)}\pi(x) +  O\!\left(
\pi(x)  M(x)
\right), 
$$
and $C_{m,n,M(x)} \le \frac{32}{\pi^2}m(2M(x))^{1/m}$, where the implied constant depend on $m$ and $n$. Therefore, for sufficiently large $x$, one has
$$
        \#\left\{ p\in (x,2x]: |a(p^m)-a'(p^n)| < M(x) \right\} = O\!\left(
        \pi(x)  M(x)^{1/m}
             \right).
$$
Substituting the value of $M(x)$ and using \eqref{eq: sympower difference estimation}, proves the desired bound.
    
{\bf{Proof of $(ii)$}}: For a fixed $x$, define
$$
  E_{m,n,M(x)}^{'} = \{ (u,v) \in [-2,2]^2 : \lvert U_m(u/2) U_n(v/2) \rvert< M(x)\}.
$$
Applying Theorem \ref{thm: sympower product} with $\e =M(x)$ and using \eqref{eq: sympower product estimation}, a similar argument as before completes the proof.

\section{Proof of Corollary \ref{cor: least prime results}}\label{sec: proof of least prime}
We present a detailed proof for the existence of prime $p_1$ such that $|a(p_1^m)-a'(p_1^n)|<\e$, and outline the argument in the complementary inequality.

By \Cref{thm:sympower-diagonal1}, the number of primes $p \le x$ for which
\begin{equation}\label{eq: difference (i)}
    |a(p^m) -a'(p^m)| < \e
\end{equation}
is at least 
\begin{equation*}\label{positive}
    C_{m,n,\e} \pi(x) - d_1
\frac{\pi(x) (\log(kk'NN' \log x))^{1/4}}{(\log x)^{1/8}},
\end{equation*}
for some positive constant $d_1.$ Here the constant $C_{m,n,\e}\le 1$ is positive and depends on $m,n,\e$. 
To ensure the existence of a prime satisfying \eqref{eq: difference (i)}, it suffices to choose $x$ such that
\begin{equation}\label{eq: least prime difference eq1}
          C_{m,n,\e}  - d_1
\frac{ (\log(kk'NN' \log x))^{1/4}}{(\log x)^{1/8}}>0.
\end{equation}
Let $y = \log x$ and  $A_{m,n,\e} = \frac{2d_1^4}{C_{m,n,\e}^4}$.
 Then inequality \eqref{eq: least prime difference eq1} reduces to
$$
        \sqrt{y} > \frac{A_{m,n,\e}}{2} (\log (kk'NN') + \log y).
$$
It is therefore enough to choose $y$ so that both
\begin{equation}\label{eq:two inequalities}
        \sqrt{y} >  A_{m,n,\e} \log (kk'NN')\quad {\rm and} \quad  \sqrt{y} >  A_{m,n,\e}  \log y
\end{equation}
hold. The first inequality is satisfied whenever 
$$
    y > \left(A_{m,n,\e} ( \log (kk'NN')\right)^2.
$$
For the second inequality, we note that $\frac{\sqrt{y}}{\log y}\to \infty$ as $y\to\infty$. Thus, 
$$
\sqrt{y} > A_{m,n,\e} \log y \quad \text{for} \quad  y= O(1),
$$
where the implied constant depends on $m,n,\e$. Consequently, both the inequalities in \eqref{eq:two inequalities} are satisfied whenever  $y >  C_1 (\log(kk'NN'))^2$ for a suitable constant $C_1$ depending only on $m$, $n$, and $\e$. Therefore, inequality \eqref{eq: least prime difference eq1} holds for
$$
           x > (kk'NN')^{C_1\log kk'NN'}.
$$
 which gives the existence of a prime $p_1 <  (kk'NN')^{C_1\log kk'NN'}$ such that 
 $$
         |a(p_1^{m})- a'(p_1^n)| < \e.
 $$ 
 For the second assertion,   \Cref{thm:sympower-diagonal1} implies that the number of primes $p \le x$ satisfying
\begin{equation}\label{eq: difference (ii)}
    |a(p^m) -a'(p^m)| > \e
\end{equation}
is at least 
\begin{equation*}\label{positive}
  ( 1- C_{m,n,\e})\pi(x) - d_2
\frac{\pi(x) (\log(kk'NN' \log x))^{1/4}}{(\log x)^{1/8}},
\end{equation*}
for some positive constant $d_2$. Therefore, to ensure the existence of a prime $p$ such that \eqref{eq: difference (ii)} holds, it is enough to find $x$ such that the above inequality is positive, and thus arguing as before, we complete the proof.
\section{Proof of Theorem \ref{thm:density-ratio}}\label{sec: proof of density in R}
For integers $m,n>0$, $t\in \R$, and $\e>0$, define
$$
E_{m,n,\e}^{t} = \left\{(u,v)\in[-2,2]^2: U_{n}(v/2)\neq 0, \,  \left\lvert\frac{U_m(u/2)}{U_n(v/2)}-t\right\rvert < \e \right\}.
$$ 
Then by relations \eqref{eq:hecke relation}, for every $p\nmid NN'$ such that $a'(p^n)\neq 0$, we have
$$
    \left\lvert \frac{a(p^m)}{a'(p^n)}-t   \right\rvert<\e \iff (a(p),a'(p))\in E_{m,n,\e}^{t}.
$$
The boundary of the set $E_{m,n,\e}^{t}$ is contained in the union of the curves
$$
    U_n(v/2)=0, \quad  U_m(u/2)-(t-\e)U_n(v/2)=0, \quad   U_m(u/2) - (t+\e) U_n(v/2)=0, 
$$
along with the lines $u=\pm 2$ and $v=\pm 2$. Thus, applying {\cite[Lemma 5.1]{kkm2}} to these curves,  we obtain
$$
    \partial E_{m,n,\e}^{t} =  \gamma_1\cup \cdots \cup \gamma_{\alpha},
$$
where each $\gamma_i$ is  a continuous curve,  and  $\alpha$ depends on $m$ and $n$. Moreover, if $\gamma_i$ is not a vertical line segment, then for each $a\in [-2,2]$,
$$
\#\{\gamma_i \cap \{(u,v):u=a\}\} \le \max(m,n).
$$
Applying {\cite[Lemma 5.1]{kkm2}}, we obtain that $L = {\rm length}(\partial E_{m,n,\e}^{t})$ depends only on  $m$ and $n$. Consequently, \Cref{thm:main1} yields
$$
        \#\!\!\left\{p\le x: a'(p^n)\neq 0,  \left\lvert \frac{a(p^m)}{a'(p^n)}-t   \right\rvert\!<\!\e\right\}\!=\!\mu_{\rm JST}(E_{m,n,\e}^{t})  \pi(x)+ O\!\left(
\frac{\pi(x) (\log(kk'NN' \log x))^{1/4}}{(\log x)^{1/8}}
\right),
$$
where the implied constant depends on $m$ and $n$. Writing 
$$
            F_{m,n,\e}^{t} \coloneqq \mu_{\rm JST}(E_{m,n,\e}^{t}),
$$
we see that the constant $F_{m,n,\e}^{t} $ is effectively computable. It therefore remains to show that $F_{m,n,\e}^{t} >0$, or equivalently that,   $E_{m,n,\e}^{t} $ has a positive joint Sato--Tate measure. To this end, we show that $E_{m,n,\e}^{t} $ contains a non-empty open ball, hence has a positive joint Sato--Tate measure. 

Choose  $\kappa>0$ such that
$$  
     \kappa < \frac{1}{|t|+\e+1}.           
$$
Then $|\kappa t| < 1$ and $|\kappa|<1$. Since
$$[-1,1] \subset U_m([-1,1]) \quad \text{and} \quad [-1,1] \subset U_n([-1,1]),
$$
there exists $(u_0,v_0)\in [-2,2]^2$ such that
$$
        U_m(u_0/2) = \kappa t\quad {\rm and}  \quad  U_n(v_0/2) = \kappa.
$$
In particular,
$$
       U_n(v_0/2)\neq 0 \quad \text{and} \quad     \frac{ U_m(u_0/2) }{U_n(v_0/2)} = t.
$$
Therefore, the function $ (u,v) \mapsto \frac{ U_m(u/2) }{U_n(v/2)}$ is continuous at $(u_0,v_0)$ and there exists an open ball $\mathcal N_r((u_0,v_0))$  of radius $r>0$ centred at $(u_0,v_0)$ such that
$$
        \left\lvert \frac{ U_m(u/2) }{U_n(v/2)}-t  \right\rvert <\e, \quad \text{for all } (u,v) \in \mathcal{N}_r((u_0,v_0)).
$$
Therefore $\mathcal{N}_r((u_0,v_0)) \subset E_{m,n,\e}^{t} $, and this completes the proof. 
\section{Concluding Remarks} \label{sec: concluding remarks}
The results obtained in this paper reveal some interesting aspects of the arithmetic relations among the Fourier coefficients of two twist-inequivalent newforms.  We know that the density of primes $p$ for which
$ a(p^m)-a'(p^n)=0$ or $ a(p^m)a'(p^n)=0$ is zero. In contrast, Theorems ~\ref{thm:sympower-diagonal1} and \ref{Atkin} show that, for every fixed $\varepsilon>0$, each of the inequalities
$$
|a(p^m)-a'(p^n)|<\varepsilon
\quad\text{and}\quad
|a(p^m)a'(p^n)|<\varepsilon
$$
hold for a positive density of primes. In particular,
$$
\liminf_{p\to\infty}|a(p^m)-a'(p^n)|=0,
$$
so the coefficients come arbitrarily close to each other infinitely often. On the other hand, replacing the fixed constant $\varepsilon$ by the quantity
\[
\frac{(\log\log p)^{1/4}}{(\log p)^{1/8}},
\]
reduces the corresponding sets of primes to density zero. Thus, allowing a fixed error and one that tends to zero leads to fundamentally different distributional behaviour.

The behavior of the ratios differs in nature  and they are quantitatively dense in $\mathbf{R}$, in the sense that every non-empty open interval contains ratios of Fourier coefficients for a positive density of primes. Hence, the existence of a non-empty open interval containing ratios of Fourier coefficients for only a density-zero set of primes forces the two newforms to be twist-equivalent. Taken together, our results demonstrate that not only exact algebraic relations, but also the closeness of Fourier coefficients, their individual sizes, and distribution of their ratios can be used to determine the newforms.

\section*{Acknowledgement}
AK was supported by ANRF under the Research Grant ANRF/ARGM/2025/000494/MTR. MK was supported by SEED Grant SGT-100115.
 \bibliography{reference}
	\bibliographystyle{alpha}

\end{document}